\documentclass[preprint,12pt]{elsarticle}

\usepackage{amssymb}
\usepackage{amsmath}
\usepackage{amsfonts}
\usepackage{graphicx}
\usepackage{graphics}
\usepackage{tikz}
\usepackage{titlesec}
\usepackage[english]{babel}

\newtheorem{theorem}{Theorem}

\newtheorem{corollary}[theorem]{Corollary}

\def\qed{\vbox{\hrule
 \hbox{\vrule\hbox to 5pt{\vbox to 8pt{\vfil}\hfil}\vrule}\hrule}}

\usepackage{tikz}
\usepackage{tkz-graph}

\journal{xxxxxx}

\begin{document}
\begin{frontmatter}

\title{New lower bounds for the Estrada and Signless Laplacian Estrada Index of a Graph}

\author[]{Juan L. Aguayo, Juan R. Carmona\corref{cor1}}
\ead{juanaguayo@uach.cl, juan.carmona@uach.cl}
\address{Instituto de Ciencias F\'isicas y Matem\'aticas,
  Universidad Austral de Chile, Independencia 631 - Valdivia - Chile.}
\cortext[cor1]{Corresponding author}

\author{Jonnathan Rodr\'{\i}guez}
\ead{jrodriguez01@ucn.cl}
\address{Departamento de Matem\'{a}ticas, Facultad de Ciencias B\'{a}sicas, Universidad de Antofagasta, Av Angamos 601, Antofagasta, Chile.}

\begin{abstract}
Let $G$ be a graph on $n$ vertices and $\lambda_1,\lambda_2,\ldots,\lambda_n$ its  eigenvalues. The Estrada index of
$G$ is defined as $EE(G)=\sum_{i=1}^n e^{\lambda_i}.$ In this paper, using a different demonstration technique, new lower bounds are obtained for the Estrada index and Signless Laplacian Estrada Index, that depends only the number of vertices, the number of edges and the energy of the graph is given.
Moreover, lower bounds for the Estrada index of an symmetric nonnegative matrix are established.
\end{abstract}


\end{frontmatter}
\section{Introduction}

Throughout the paper, we consider G an arbitrary connected graph with edge set denoted by $\mathcal{E}(G)$ and its vertex set $V(G)=\{ 1, \ldots, n \}$ with cardinality $m$ and $n$, respectively. We say that $G$ is an $(n,m)$ graph.
If $ e \in \mathcal{E}(G)$ has end vertices $i $ and $j$, then we say that $i$ and $j$ are adjacent and this edge is denoted by $ij$. For a finite set $U$, $|U|$ denotes its cardinality.
Let $K_n$ be the complete graph with $n$ vertices and $\overline{K_n}$ its (edgeless) complement. A graph $G$ is bipartite if there exists a partitioning of $V(G)$ into disjoint, nonempty sets $V_1$ and $V_2$ such that the end vertices of each edge in $G$ are in distinct sets $V_1$, $V_2$. A graph $G$ is a complete bipartite graph if $G$ is bipartite and each vertex in $V_1$ is connected to all the vertices in $V_2$. If $|V_1|=p$ and $|V_2|=q$, the complete bipartite graph is denoted by $K_{p,q}$. For more properties of bipartite graphs, see \cite{haemers}.\\
	
	
The adjacency matrix $A(G)$ of  the  graph $G$ is  a 
symmetric matrix of order $n$ with entries $a_{ij}$, such that $a_{ij}=1$ if $ij \in \mathcal{E}(G)$ and $a_{ij}=0$ otherwise. Denoted by $\lambda_1\geq \ldots \geq \lambda_n$ to the eigenvalues of $A(G)$,  
see \cite{C-D-S1, C-D-S2}.\\
The matrix $Q(G)=D(G)+A(G)$ is called the signless Laplacian matrix of $G$
\cite{m2,m3,m4}, where $D(G)$ is the diagonal matrix of vertex degrees of $G$. 
It well know that the matrix $Q(G)$ is positive semidefinite, see \cite{m46}. We denoted the eigenvalues of the signless Laplacian matrix by $q_{1}(G)\geq \ldots \geq q_{n-1}(G) \geq q_{n}(G)\geq 0.$\\
    
 A matrix is singular if it has zero as an eigenvalue, otherwise, it is called non-singular. A graph $G$ will be said non-singular if its adjacency matrix is non-singular.\\
    

The Estrada index of the graph $G$ is defined as
$$EE(G)= \sum_{i=1}^n e^{\lambda_i}.$$
This spectral quantity is put forward by E. Estrada \cite{E1} in the year 2000. There have been found a lot of chemical and physical applications, including quantifying the degree of folding of long-chain proteins, \cite{E1, E2, E3, G-F-G-M-V, G-R-F-M-S, G-G-M-S}, 
and complex networks \cite{E4, E5, S1, S2, S3, S4}. 
Mathematical properties of this invariant can be found in e.g. \cite{F-A-G, G-D-R, G-R, K, S5, Z, Z-Z}. \\

In \cite{A}, was introduced the innovative notion of the signless Laplacian Estrada index of a graph as $$EE(Q(G))=SLEE(G)= \sum_{i=1}^n e^{q_i}.$$ Moreover, the authors established new lower and upper bounds for $EE(Q(G))$ in terms of the number of vertices and edges for this invariant.\\

Denote by $M_k = M_k(G)$ to the $k$-th spectral moment of the graph $G$, i.e., $$M_k= \displaystyle \sum_{i=1}^n(\lambda_i)^k.$$ Then, we can write the Estrada index as $$EE(G) = \sum_{k=0}^{\infty}\frac{M_k}{k!}.$$ In \cite{C-D-S1}, for an $(n, m)$-graph $G$, the authors proved that

\begin{equation}\label{eq1}
	M_0 =n,\,\, M_1 =0,\,\, M_2 =2m,\,\, M_3 =6t,
\end{equation}
where $t$ is the number of triangles in $G$.\\
	
\noindent The energy of $G$ was defined by I. Gutman in 1978, \cite{G1}, as $$E(G)= \sum_{i=1}^n|\lambda_i|.$$
Similarly, the signless Laplacian energy of a graph $G$ with $n$ vertices and $m$ edges is defined in \cite{m8} by

\begin{equation}
    QE(G)=\sum_{i=1}^{n}\Big\arrowvert q_{i}-\frac{2m}{n}\Big\arrowvert.  
\end{equation}

\noindent Nowadays studies on the signless Laplacian energy are still in course \cite{m32,m30}.
    
\noindent In \cite{K-M}, Koolen and Moulton showed that the following relation holds for all graph $G$

\begin{equation}\label{1}
	E(G) \leq \lambda_1+\sqrt{(n - 1)(2m-\lambda_1^2)}.
\end{equation}
	
For all $(n,m)$-graph $G$ connected and nonsingular, Das et al. in \cite{D-M-G}, proved the following relation holds
	
\begin{equation}\label{2}
	E(G) \geq \lambda_1 +(n-1) +\ln(det A)+\ln(\lambda_1),
\end{equation}
	
then using the inequality $2m/n \leq \lambda_1$, they obtained the following upper bound 	$$E(G)\geq \frac{2m}{n} + (n-1) + \ln | detA |+\ln \left(\frac{2 m}{ n}\right ).$$ 	With respect to Estrada index, for all $(n,  m)$-graph, De  la  Pe\~na  et  al. in \cite{P-G-R} proved that

	\begin{equation}\label{123}
	    \sqrt{n^2+4m}\leq EE(G) \leq n-1 + e^{\sqrt{2m}}.
	\end{equation}

 Equality  on  both  sides  of (\ref{123}) is  attained if  and  only  if $G$ is isomorphic to $\overline{K_n}$.\\
	
\noindent	In this paper, under the motivation and ideas in \cite{D-M-G}, we obtain new lower bounds for the Estrada index and Estrada Index of the Signless Laplacian matrix in terms of the number of vertices, edges, and energy of an arbitrary G graph. Furthermore, lower bound for $ EE (G) $ considering an arbitrary symmetric nonnegative matrix are established.\\

\section{Lower bound for the Estrada index}

In this section, new lower bounds for Estrada index in relation to the number of vertices and edges of a given graph are established.
	
\begin{theorem} 
		Let $G$ a $(n, m)$-graph  then,
		\begin{equation}\label{cota0}
		EE(G)\geq e^{\left( \frac{2m}{n}\right)}+(n-1)-\frac{2m}{n}.
		\end{equation}
		Equality holds in (\ref{cota0}) if and only if G is isomorphic to $\overline{K_n}$.
\end{theorem}
\begin{proof}
		Consider the following function 
		\begin{equation}\label{f}
		f(x)= (x-1)-\ln(x), \, \, x>0.
		\end{equation}
		Clearly the function $f$ is decreasing in $(0,1]$ and increasing in $[1,+\infty)$
		Consequently, $f(x) \geq f(1) = 0$, implying that 
		\begin{equation}\label{eq7}
		x \geq 1 + \ln x, \,\,\,x > 0,
		\end{equation}
		the equality holds if and only if $x = 1$. Let $G$ be a graph of order $n$, using (\ref{eq1}) and (\ref{eq7}), we get:
		\begin{equation}\label{Energia0}
		\begin{array}{lllll}
		EE(G) & \geq & e^{\lambda_1} + (n-1)+ \displaystyle \sum_{k=2}^n \ln e^{\lambda_k}  \\
		& = & e^{\lambda_1} + (n -1) +\displaystyle \sum_{k=2}^n \lambda_k \\
		& = &  e^{\lambda_1} + (n -1) +M_1 - \lambda_1\\
		
		& = &  e^{\lambda_1} + (n -1)  - \lambda_1.
		\end{array}
		\end{equation}
		Define the function
		\begin{equation}\label{fi}
		\phi(x)= e^x + (n-1) - x,\qquad x>0.
		\end{equation}
		
		\noindent Note than, this is an increasing function on $D_{\phi}=[0,+\infty)$.\\
		
\noindent On the other hand, we have inequality
		$$\lambda_1\geq \frac{2m}{n}\geq 0,$$
		
\noindent therefore, we verify
	    $$\phi\left(\lambda_1\right)\geq \phi\left(\frac{2m}{n}\right).$$
	
\noindent Finally, we get

\begin{equation}\label{J}
EE(G)\geq e^{\left( \frac{2m}{n}\right)}+(n-1)- \left(\frac{2m}{n}\right)=J(G).    
\end{equation}

Suppose now that the equality holds in (\ref{cota0}). Then all the inequalities in (\ref{Energia0}) must be considered as equalities. From the equality (\ref{eq7}), we get $e^{\lambda_2}=\ldots=e^{\lambda_n}=1,$ then $\lambda_2=\ldots={\lambda_n}=0$ and $\lambda_1=\frac{2m}{n}$. It follows that $G$ is regular. Thus $G$ is isomorphic to the $\overline{K_n}.$ \cite{Me00}
		\end{proof}
\vspace{0.5cm}
A direct consequence is the following result.

	\begin{corollary}
	    Let $G$ be a connected $\alpha$-regular graph with $n$ vertices. Then,
		\begin{equation}
		EE(G)\geq e^{\alpha}+n-\alpha-1.
		\end{equation}
	\end{corollary} 
	
In the following result, we obtain a sharp lower bound of the Estrada index for a bipartite graph.
	
	\begin{theorem}
		Let $G$ be a bipartite $(n,m)$-graph, with $n\geq2$, then
		\begin{equation}\label{cota1}
		EE(G)\geq 2\cosh{\left(\frac{2m}{n}\right)}+(n-2).		
		\end{equation}
Equality holds in (\ref{cota1}) if and only if G is isomorphic to $K_{\sqrt{m},\sqrt{m}}$, with $2\sqrt{m}=n$.
	\end{theorem}	
	
	\begin{proof}
	Considering (\ref{eq1}) and (\ref{eq7}), we obtain
	\begin{equation}\label{EE2}
	\begin{array}{lllll}
	EE(G) & = & e^{\lambda_1}+e^{-\lambda_1} + \displaystyle \sum_{k=2}^{n-1} e^{\lambda_k}  \\
	& \geq & 2\cosh{\lambda_1} + (n-2) + \displaystyle \sum_{k=2}^{n-1} \lambda_k \\
	& = &  2\cosh{\lambda_1} + (n-2) + M_1 + \lambda_1 - \lambda_1\\
	& = &   2\cosh{\lambda_1} +(n-2).
	\end{array}
	\end{equation}	
	Since, $\Phi(x)= 2\cosh{x} + n-2,$
	%
	is an increasing function on $D_{\Phi}=[0,+\infty)$. \\
	
Analogous to Theorem 1, we have inequality 
	
$$\lambda_1\geq \frac{2m}{n}\geq 0,$$

\noindent therefore, we verify
	    $$\Phi\left(\lambda_1\right)\geq \Phi\left(\frac{2m}{n}\right).$$
	
\noindent Finally, we get
	
\begin{equation}\label{JB}
    EE(G)\geq 2\cosh{\left(\frac{2m}{n}\right)}+(n-2)=JB(G).
\end{equation}	

Suppose now that the equality holds in (\ref{cota1}). Then all the inequalities in (\ref{EE2}) must be considered as equalities. From the equality (\ref{eq7}), we get $e^{\lambda_2}=\ldots=e^{\lambda_{n-1}}=1,$ then $\lambda_2=\ldots={\lambda_{n-1}}=0$ and $\lambda_1=-\lambda_n=\frac{2m}{n}$, which implies that $G$ is regular with three distinct eigenvalues, then $G \cong K_{\sqrt{m},\sqrt{m}}$, with $2\sqrt{m}=n$.. (see \cite{K-M2}.) 

	\end{proof}
	\vspace{0,3cm}
	The next result is obtained applying the proof technique of the above Theorem to signless Laplacian matrix of graph $G$.

\begin{theorem} 

		Let $G$ a $(n, m)$-graph  then,
		\begin{equation}\label{cota}
		EE(Q(G))\geq e^{\left( \frac{4m}{n}\right)}+(n-1)+2m-\frac{4m}{n}
		\end{equation}
		Equality holds in (\ref{cota}) if and only if G is isomorphic to $\overline{K_n}$.
	\end{theorem}
	
	\begin{proof}
        Using (\ref{eq7}), we get:
        \begin{equation}\label{Energia}
		\begin{array}{lllll}
		EE(Q(G)) & = & e^{q_1} +  \displaystyle \sum_{k=2}^n e^{q_k}  \\
	    & \geq & e^{q_1} + (n -1) +\displaystyle \sum_{k=2}^n \ln(e^{q_k}) \\
		& = & e^{q_1} + (n -1) +\displaystyle \sum_{k=2}^n q_k \\
		& = &  e^{q_1} + (n -1) +\displaystyle \sum_{k=1}^n q_k - q_1\\
		
		& = &  e^{q_1} + (n -1)+2m  - q_1.
		\end{array}
		\end{equation}
		Thereby, the function $\tau(x)= e^x + n-1+2m - x,x>0,$
		%
		is an increasing function on $D_{\tau}=[0,+\infty)$.\\
		
		\noindent In \cite{Feng}, is showed the inequality
		$q_1\geq 2\lambda_1$ and join with  $\lambda_1 \geq \frac{2m}{n},$
		we verify
	    $$\tau\left(q_1\right)\geq \tau\left(\frac{4m}{n}\right).$$
	
		\noindent Finally, we get
		$$ EE(Q(G))\geq e^{\left( \frac{4m}{n}\right)}+(n-1)+2m- \left(\frac{4m}{n}\right).$$
		Suppose now that the equality holds in (\ref{Energia}). From the equality in (\ref{eq7}), we get $e^{q_2}=\ldots=e^{q_n}=1,$ then $q_2= \ldots={q_n}=0$  and $q_1=\frac{2m}{n}$. It follows that $G$ is regular. Thus $G$ is isomorphic to the $\overline{K_n}.$ 
		\end{proof}
\vspace{0.4cm}

\subsection{Estrada Index of a symmetric nonnegative  matrix} \vspace{0.4cm}

In view of the demonstration technique used in the previously to obtain lower bounds, we will apply this to the Estrada Index of a given arbitrary Hermitian matrix.\\

A symmetric nonnegative matrix $M=\left( m_{ij}\right) $ of order $n $. Let $\rho_{1},\rho_{2},\dots,\rho_{n}$ be the eigenvalues of $M$ and consider $\left\vert \rho_{i}\right\vert \leq \rho_{1}$, for all $i=1,\ldots,n$.\\

The next result shows bounds for the greater eigenvalue of the matrix in relation with your rows, see \cite{M}.

\begin{theorem}
    If $M$ is symmetric nonnegative matrix with maximal eigenvalue $\lambda $ and row sums
    $s_{1},\dots ,s_{n}$. Then%
    \begin{equation}
    r\leq \lambda \leq R  \label{RSIM}
    \end{equation}%
    where $r=\min \left\{ s_{i}:i=1,\dots ,n\right\} $ and $R=\max \left\{
    s_{i}:i=1,\dots ,n\right\} .$ If $M$ is irreducible, then equality can hold on
    either side of (\ref{RSIM}) if and only if all row sums of $M$ are equals.
\end{theorem}

Applying the above proof technique to the matrix $M$ and considering the inequality $x\geq 1+\ln \left( x\right) $, we have%

\begin{eqnarray*}
EE\left( M\right)  &=&e^{\rho _{1}} +\sum_{i=2}^{n} e^{\rho _{i}}  \\
&\geq & e^{\rho _{1}} +(n-1) +\sum_{i=2}^{n}\rho
_{i} \\
&=& e^{\rho _{1}} +(n-1) +\sum_{i=1}^{n}\rho
_{i}-\rho _{1} \\
&=& e^{\rho _{1}} + (n-1) +Tr(M)-\rho _{1},
\end{eqnarray*}%
where $Tr(M)$ is the trace of $M$. Define the function $$\psi \left( x\right) = e^x +
(n-1) +Tr(M)-x.$$ Its straightforward verified that $\psi \left(
x\right) $ is increasing in $[0,\infty )$. Then, by above Theorem $\rho
_{1}\geq \min \left\{ s_{i}:i=1,\dots ,n\right\} =r.$ Since $\psi \left(
\rho _{1}\right) \geq \psi \left( r\right) $, we proved the following result.

\begin{theorem}
Let $M$ be a nonnegative symetric matrix of order $n$ with eigenvalues $\rho _{1},\rho
_{2},\dots ,\rho _{n}$ and $r=\min \left\{ s_{i}:i=1,\dots ,n\right\}.$ Then%
\[
EE\left( G\right) \geq e^r +Tr\left(M\right)+(n-1)-r.
\]
\end{theorem}

\section{Estrada index and energy}

In this section, new lower bounds for the Estrada index in relation to the energy of the graph $G$ are established.
	
	\begin{theorem}\label{Eny}
		Let $G$ be a $\left( n,m\right) $-graph with $k$ non-negative eigenvalues. Then%
\begin{equation}\label{co1}
    EE\left( G\right) \geq \frac{E\left( G\right) }{2}+e^{ \lambda
_{1}} +(k-1)-\lambda _{1}.
\end{equation}

	Equality holds in (\ref{co1}) if and only if G is isomorphic to $\overline{K_n}$.
	\end{theorem}
	
	\begin{proof}
		Let $x\geq 0$, consider the following function%
\begin{equation}\label{func021}
    g\left( x\right) =-1-x+e^x  .
\end{equation}%
the equality holds if and only if $x = 0$. Is straightforward show that function $g\left( x\right) $ is increasing in $[0,+\infty)$. Then $g\left( x\right) \geq g\left( 0\right) $, implying that%
\begin{equation}
x\leq e^x -1,x\geq 0.  \label{nlb}
\end{equation}

Note that, $A\left( G\right) $ is a symmetric matrix with zero trace, these
eigenvalues are real with sum equal to zero, i.e,
\begin{equation}
\lambda_{1}\geq \ldots \geq \lambda_{n}   \label{trivial1}
\end{equation}
and
\begin{equation}
\lambda_{1}+\ldots
+\lambda_{n}=0.  \label{trivial}
\end{equation}

Then by the definition of energy join to (\ref{trivial1}) and (\ref{trivial}) we have
\begin{equation}
\frac{E\left( G\right) }{2}=\sum_{\lambda _{i}>0}\lambda
_{i}^{+}=-\sum_{\lambda _{i}<0}\lambda _{i}^{-}.  \label{eqEngy}
\end{equation}

\noindent Supposed that $A\left( G\right)$ have $k$ non-negative eigenvalues, then using (\ref{eqEngy}) and (\ref{nlb}) we obtain

\begin{eqnarray*} \label{enr11}
\frac{E\left( G\right) }{2} &=&\sum_{i=1,\lambda _{i}\geq0}^{k}\lambda _{i}  \\ 
&=&\lambda _{1}+\sum_{i=2,\lambda _{i}\geq0}^{k}\lambda _{i} \\
&\leq &\lambda _{1}+\sum_{i=2,\lambda _{i}\geq0}^{k}\left( e^{\lambda_{i}} -1\right)  \\ 
&=&\lambda _{1}-\left( k-1\right) +\sum_{i=1,\lambda _{i}\geq 0}^{k} e^{\lambda_{i}} - e^{\lambda _{1}}\\
&\leq &\lambda _{1}-\left( k-1\right) +\sum_{i=1,\lambda _{i}\geq0}^{k} e^{\lambda _{i}} +\sum_{i=k+1,\lambda _{i}<0}^{n} e^{\lambda _{i}} - e^{\lambda _{1}}  \\
&=&\lambda _{1}-\left( k-1\right) +\sum_{i=1}^{n} e^{\lambda_{i}} - e^{\lambda _{1}}.%
\end{eqnarray*}

Thereby, we obtain the first result.\\

Suppose now that the equality holds. From the equality in (\ref{func021}), we get $q_1=\ldots=q_n=0.$ Then $k=n$. Therefore $G$ is isomorphic to the $\overline{K_n}.$
	\end{proof}
    A direct consequence of the above Theorem is the following.

\begin{corollary}
		Let $G$ be a $\left( n,m\right) $-graph with $k$ non-negative eigenvalues. Then%
\begin{equation}\label{col1}
    EE\left( G\right) \geq \frac{E\left( G\right) }{2}+ e^{\left(\frac{2m}{n}\right)}+(k-1)- \frac{2m}{n}.
\end{equation}

	Equality holds in (\ref{col1}) if and only if G is isomorphic to $\overline{K_n}$.
	\end{corollary}

	Considering the above Theorem and the lower bound due Das et al in (\ref{2}), we obtain the following result.
	
	\begin{corollary}
		Let $G$ be a connected non-singular graph of order $n$ with $k$ strictly positive eigenvalues. Then%
\begin{eqnarray*}
EE\left( G\right) \geq \frac{1}{2}\left( n-1+\ln \left( \det \left( A\left(
G\right) \right) \right) +\ln \left( \lambda _{1}\right) \right) +e^{\lambda _{1}} +(k-1)-\frac{\lambda _{1}}{2}.
\end{eqnarray*}
	\end{corollary}
	
Applying the demonstration technique used in above Theorem, to the signless Laplacian matrix of a graph G, we obtain the following result.

\begin{theorem}
Let $\left( n,m\right) $ graph $G$. Then%
\[
EE\left( Q\left( G\right) \right) \geq QE\left( G\right) +e^{q_{1}}
 + (n-1) -q_{1}.
\]%
The equality holds if and only if $G=\overline{K_{n}}$.
\end{theorem}

\begin{proof}
By the definition of signless Laplacian energy, we have%
\begin{eqnarray*}
QE\left( \left( G\right) \right) &=&\sum_{i=1}^{n}\Big\arrowvert q_{i}-\frac{%
2m}{n}\Big\arrowvert\leq \sum_{i=1}^{n}q_{i} \\
&=&q_{1}+\sum_{i=2}^{n}q_{i} \\
&\leq &q_{1}+\sum_{i=2}^{n}\left( e^{ q_{i}} -1\right) \\
&=&q_{1}+\sum_{i=1}^{n}e^{ q_{i}} -\left( n-1\right) -e^{q_{1}}
 \\
&=&q_{1}+EE(Q(G)) -\left( n-1\right) -e^{q_{1}}
.
\end{eqnarray*}%
Suppose now that the equality holds. From the equality in (\ref{func021}), we get $q_1=\ldots=q_n=0.$ Therefore $G$ is isomorphic to the $\overline{K_n}$. Therefore, we prove the result. 
\end{proof}

A direct consequence of the above Theorem is the following result.

\begin{corollary}
		Let $G$ be a $\left( n,m\right) $-graph. Then%
\begin{equation}\label{col1}
    EE\left( Q\left( G\right) \right) \geq QE\left( G\right) +e^{\left(
\frac{4m}{n}\right)} + (n-1) - \frac{4m}{n}.
\end{equation}
	Equality holds in (\ref{col1}) if and only if G is isomorphic to $\overline{K_n}$.
	\end{corollary}

\section{Comparison and Conclusions}

In this section, we present some computational experiments to compare our new lower bounds with the lower bounds existing in the literature for connected graphs.

Some examples of families of graphs where Theorem 1 is verified, more than one, it is shown that the lower bound for the Estrada index obtained in this first main results is better than the Estrada index existing in the literature.
Then, for our analysis we will use the following notation, for (\ref{cota0}) and (\ref {123}):  $$EE(G)\geq e^{\left(\frac{2m}{n}\right)}+(n-1)-\left(\frac{2m}{n} \right)=J(G),$$  $$EE(G)\geq\sqrt{n^2+4m}=CP(G).$$
\begin{enumerate}
\item {\bf Star graph}. Note that $S_n$ denote a complete bipartite graph $K_{1,k}.$ 
It is well known that $S_n$ hold the relation between its edges and vertices is given by $m=n-1$. Then, replacing this condition, we have $J(S_n)=e^{\left(2-\frac2n\right)}+\left(n+\frac2n\right)-3$ and $CP(S_n)=\sqrt{n^2+4n-4}$, then it is easily demonstrated that $J(S_n)\geq CP(S_n),$ for all $n\neq 2.$ Thereby, we obtain that $EE(S_n)\geq J(S_n)\geq CP(S_n)$, which shows that our lower bound is much better than the one proposed by (\ref{123}) for the Estrada index. \\

\item {\bf Path graph}. We will denoted to $P_n$ a path. 
Note that $P_n$ satisfies the relation $m=n-1$ which is the same relation that exists in $S_n$ between $n$ and $m$, then we shows $EE(P_n)\geq J(P_n)\geq CP(P_n),$ for all $n \neq 2 $. \\

\item {\bf Complete graph}. For $K_n$ 
is true that $m=\frac{n(n-1)}{2}$. Then we have $J(K_n)=e^{n-1}$ and $CP(K_n)=\sqrt{3n^2-2n}$. then it is readily demonstrated that $J(K_n)\geq CP(K_n),$ for all $n \neq 2.$ Thereby, we obtain that $EE(K_n)\geq J(K_n)\geq CP(K_n)$, which shows that our lower bound is much better than the one proposed by (\ref{123}) for the Estrada index.\\

\item {\bf Cycle graph}. For a cycle graph $C_n$, the condition 
$m=n$ is verified. So we have $J(C_n)=e^2+n-3$ and $CP(C_n)=\sqrt{n^2+4n}$, then we show that $J(C_n)\geq CP(C_n)$ for all $n.$ Then, $EE (C_n)\geq J(C_n)\geq CP(C_n)$, which shows that our lower bound is much better. \\

\end{enumerate}

Therefore, we will claim that, for some kinds of graphs $G$, $J(G) \geq CP(G).$\\

\end{document}